\documentclass[a4paper, 10pt]{article}
\usepackage{srcltx}
\usepackage[colorlinks=false,bookmarks=true]{hyperref}
\usepackage{amsmath}
\usepackage{amssymb}
\usepackage{amsthm}
\usepackage[latin2]{inputenc}
\usepackage[english]{babel}
\usepackage{graphics}
\usepackage{enumerate}
\usepackage{pgf}
\usepackage{url}
\allowdisplaybreaks

\newcommand{\mailto}[1]{\href{mailto:#1}{\texttt{#1}}}

\newtheorem{thm}{Theorem}[section]

\newtheorem{lem}{Lemma}[section]
\newtheorem{dfn}{Definition}[section]

\theoremstyle{definition}\newtheorem{rem}{Remark}[section]
\newcommand{\dd}[0]{\text{\protect\textup{d}}}

\newcommand{\Rr}[0]{\mathbb{R}}
\newcommand{\eps}[0]{\varepsilon}

\newcommand{\wh}[1]{\widehat{#1}}

\newcommand{\mfn}[2]{{#1}_{#2} } 
\newcommand{\mfnc}[3]{{#1}^{#2}_{#3} } 
\newcommand{\mfnh}[2]{\widehat{{#1}}_{#2} } 
\newcommand{\mfnch}[3]{\widehat{{#1}}^{#2}_{#3} } 
\newcommand{\degg}[0]{\text{\protect\textup{deg}}} 
\newcommand{\APi}[1]{\mathcal{A}^{\Pi}_{#1}} 
\newcommand{\APig}[2]{\mathcal{A}^{#1}_{#2}}

\newcommand{\BPi}[1]{B^{\Pi}_{#1}} 
\newcommand{\TPi}[1]{T^{(\Pi,#1)}} 

\newcommand{\lipg}[0]{$\text{Lip}^{\gamma}$ } 

\newcommand{\lipGP}[0]{$\text{Lip}^{\Gamma,\Pi}$ } 
{\unskip\nobreak\hskip 1em plus 1fil\nobreak$\Box$
\parfillskip=0pt%
\endtrivlist}

{\unskip\nobreak\hskip 1em plus 1fil\nobreak
\parfillskip=0pt%
\endtrivlist}

{\unskip\nobreak\hskip 1em plus 1fil\nobreak
\parfillskip=0pt%
\endtrivlist}

\begin{document}


\begin{center}
{\Large{\bf Differential equations driven by $\Pi$-rough
    paths\footnote{\textbf{A slightly revised version of the paper has been
      accepted for publication by the Proceedings of the Edinburgh
      Mathematical Society. The copyright of the published paper
      belongs to the Edinburgh Mathematical Society. The journal is
      available online:} \url{http://journals.cambridge.org/action/displayJournal?jid=PEM} }}}
\end{center}

\begin{center}
{\large Lajos Gergely Gyurk\'o}\footnote{Mathematical Institute, University of Oxford
\newline
Oxford-Man Institute of Quantitative Finance, University of Oxford
\newline
E-mail:  \mailto{gyurko@maths.ox.ac.uk}
\newline
\textbf{Keywords:} Rough paths, Differential equations, Stochastic integrals, Stochastic differential equations
\newline
\textbf{AMS 2010 Mathematics Subject Classification:} 
60H05, 
60H10, 
93C15  
\newline
This research was supported by the Oxford-Man Institute of Quantitative Finance and EPSRC.
}
\end{center}

\begin{abstract}
This paper revisits the concept of rough paths of inhomogeneous degree of smoothness (geometric $\Pi$-rough paths in our terminology) sketched by Lyons \cite{rp_paper}. Although geometric $\Pi$-rough paths can be treated as $p$-rough paths for a sufficiently large $p$ and the theory of integration of \lipg one-forms ($\gamma>p-1$) along geometric $p$-rough paths (ref. \cite{rp_paper}, \cite{rp_book}) applies, we prove the existence of integrals of one forms under weaker conditions. Moreover, we consider differential equations driven by geometric $\Pi$-rough paths and give sufficient conditions for existence and uniqueness of solution. 
\end{abstract}



\section*{Introduction}
The theory of rough paths due to Lyons \cite{rp_paper,rp_book} enables the definition of a wide class of stochastic differential equations in the path-wise sense. In particular, the rough paths representation of Brownian motion (enhanced or lifted  Brownian motion) was considered in \cite{rp_book}, and has been extensively studied by Friz \& Victoir \cite{bmFrizVictoir, bookFrizVictoir} and many others. Coutin and Qian \cite{fbmQian} proved the existence of a geometric rough path associated with the fractional Brownian motion with Hurst parameter greater than $1/4$. A different approximation of the enhanced fractional Brownian motion was studied by Millet and Sanz-Sol\'e \cite{fbmSole}. The rough path representation of an even larger class of Gaussian processes has been explored by Friz and Victoir \cite{gaussFrizVictoir, bookFrizVictoir} and others. In all these works, the roughness degree of the driving noise is described by a single real number $p$, and hence it is assumed to be homogeneous in all directions.

In section \ref{sec:PiRPs}, we revisit the definition of geometric rough paths of
inhomogeneous degree of smoothness (geometric $\Pi$-rough paths in our
terminology) sketched by Lyons \cite{rp_paper}. 
Our sharpened extension theorem of $\Pi$-rough paths specifies which terms of the signature of a rough path of inhomogeneous degree of smoothness determine the whole signature. In particular, the extension theorem determines what terms are to be specified in order to lift stochastic processes with mixed components such as Brownian and fractional Brownian components to $\Pi$-rough paths.

We note that geometric $\Pi$-rough paths can be treated as $p$-rough
paths for a sufficiently large $p$, and the theory of integration of
\lipg one-forms ($\gamma>p-1$) along geometric $p$-rough paths
(ref. \cite{rp_paper,rp_book}) could be applied. In section
\ref{sec:Integration}, we show that the
\lipg condition on the one-form can be weakened if we exploit the fact
that the underlying $\Pi$-rough path has components with roughness
parameter smaller than $p$. In particular, we introduce the definition
of \lipGP one-forms. Moreover, as the main result of the paper, we
define and prove the existence of integrals of
\lipGP one-forms along $\Pi$-rough paths (Theorem \ref{_thm_lipGP_Integration}). 


In section \ref{sec:RDE}, as a particular application of the main result, we consider differential equations of the form
\begin{equation}
dY_t=f(X_t,Y_t)dX_t, \ \ Y_0=\xi\in W
\label{eq:introRDE1}
\end{equation}
where $X$ is a geometric $\Pi$-rough paths defined on some Banach space $V$, $f:V\oplus W\to L(V,W)$ and $W$ is some Banach space. 
%
%
This equation can be rewritten as follows. 
\begin{equation}
d\tilde{Y}_t=\tilde{f}(\tilde{Y}_t)dX_t, \ \ \tilde{Y}_0=(X_0,\xi)\in V\oplus W,
\label{eq:introRDE2}
\end{equation}
where $\tilde{f}:V\oplus W\to L(V,V\oplus W)$ defined by
\begin{equation*}
\tilde{f}(v,w)(u)=(u,f(v,w)(u)).
\end{equation*}
From Lyons' Universal Limit Theorem (ref. \cite{rp_paper,rp_book,rp_notes}), we know that solution to \eqref{eq:introRDE2} exists and is unique if $f$ is \lipg with $\gamma>p$. 
We adapt the proof of \cite{rp_notes}, and show the existence and uniqueness of solution for our case under 
sufficient conditions that are weaker than that is required by Lyons' Universal Limit Theorem in the homogeneous case. 

We note that the $\Pi$-rough paths form a more general class than the class of $(p,q)$-rough paths as defined in \cite{pqrp}; in the case of $\Pi$-rough paths, apart from assuming the  parameters that specify the roughness in the various directions to be at least $1$, we do not make any further restrictions on them.   

Moreover,  throughout the paper, we assume the spatial structure of the inhomogeneity to be static. However, in general applications, the directions of the inhomogeneity of the driving noise and hence that of the solution may change in time and/or in space. Such applications fall beyond the scope of the paper, although our results might be relevant for them. 




\section{$\Pi$-rough paths}\label{sec:PiRPs}
Throughout in this section, $k$ denotes a fixed positive integer and $\Pi=(p_1,\dots,p_k)$ is a \emph{real $k$-tuple}, such that $p_i\ge 1$ is a real number for all $i\in\{1,\dots, k\}$. Furthermore, let a Banach space $V$ of the form $V=V^1\oplus\cdots\oplus V^k$ be given for some Banach spaces $V^1,\dots, V^k$.

\begin{dfn}\label{_dfn_PiBasics}
 We say that $R=(r_1,\dots,r_l)$ is a \emph{$k$-multi-index} if $1\le r_j\le k$ is an integer for all $j\in\{1,\dots,l\}$. The empty multi-index is denoted by $\epsilon$ and the set of all $k$-multi-indexes of finite length is denoted by $\mathcal{A}^{k}$. 

Given the multi-index $R=(r_1,\dots,r_l)$, we define the $k$-multi-index $R-$ by
\begin{equation*}
R-=(r_1,r_2,\dots,r_{l-1},r_l)-=(r_1,r_2,\dots,r_{l-1}).
\end{equation*}

The concatenation of the multi-indices $R=(r_1,\dots,r_l)$ and $Q=(q_1,\dots,q_m)$ is denoted by
\begin{equation*}
R\ast Q=(r_1,\dots,r_l)\ast(q_1,\dots,q_m)=(r_1,\dots,r_l,q_1,\dots,q_m).
\end{equation*}
\end{dfn}

\begin{dfn}\label{_dfn_degrees}
For the $k$-multi-index $R=(r_1,\dots,r_l)$ we denote the \emph{length} by $\Vert R\Vert=l$. Furthermore, we define the 
function $n_j$ for $j\in\{1,\dots,k\}$ by
\begin{equation*}
n_j(R):=\text{\protect\textup{card}}\{i| r_i=j, r_i\in R\}.
\end{equation*}

We introduce the $\Pi$-degree of $R$ as
\begin{equation*}
\degg_{\Pi}(R)=\sum_{j=1}^k\frac{n_{j}(R)}{p_j}.
\end{equation*}
Note that $\degg_{\Pi}(\epsilon)=0$.
We also introduce the function $\Gamma_{\Pi}:\mathcal{A}^k\to [0,\infty)$ by
\begin{equation*}
\Gamma_{\Pi}(R)=\left(\frac{n_1(R)}{p_1}\right)!\cdots\left(\frac{n_k(R)}{p_k}\right)!, \text{ for } R\in\mathcal{A}^k,
\end{equation*}
where $(\cdot)!$ denote the $\Gamma$-function. 

Let $s\ge 0$ be real. We introduce the set of $k$-multi-indices
\begin{equation*}
\APi{s}:=\left\{ 
R=(r_1,\dots,r_l)\Big|\ l\ge 1, \ \degg_{\Pi}(R)\le s \right\}.
\end{equation*}

\end{dfn}

Let $S^{\Pi}$ denote the set
\begin{equation*}
S^{\Pi}=\left\{s=\degg_{\Pi}(R)\ | \ R\in\mathcal{A}^k\right\}.
\end{equation*}
Note that $S^{\Pi}$ is unbounded from above and closed under addition. Also note that since for any $R\in\mathcal{A}^k$,
\begin{equation*}
\degg_{\Pi}(R)>\frac{\Vert R\Vert}{\max_{1\le i\le k}{p_i}}
\end{equation*}
the set $\{R\in\mathcal{A}^k|\degg_{\Pi}(R)\le s\}$ is finite for all $s\ge 0$. This implies, that the elements of $S^{\Pi}$ can be listed in ascending order. The $m$th element in the ordered $S^{\Pi}$ will be denoted by $s_m$.

\begin{dfn}\label{_dfn_InhomSetup}
The space of formal series of tensors of $V$ is equivalently represented by 
\begin{equation*}
T(V)=\bigoplus_{n=0}^{\infty}V^{\otimes n}=\bigoplus_{(r_1,\dots,r_l)\in\mathcal{A}^k}V^{r_1}\otimes\cdots\otimes V^{r_l},
\end{equation*}
where $V^{\otimes 0}:=\Rr$. 

For a $k$-multi-index $R=(r_1,\dots,r_l)$, we introduce the notation
\begin{align*}
V^{\otimes R}&=V^{r_1}\otimes\cdots \otimes V^{r_l}, \text{ and}\\
V^{(\Pi,s)}&=\sum_{\degg_{\Pi}(R)=s}V^{\otimes R}, \text{ for } s\in S^{\Pi}.
\end{align*}
In general, for a vector space $U=A\oplus B$, $\pi_A$ and $\pi_B$ denote the canonical projection onto $A$ and $B$ respectively, i.e. for $u=a+b\in U$,
such that $a\in A$ and $b\in B$, $\pi_Au=a$ and $\pi_Bu=b$.
We extensively use the projection $\pi_V$ onto the $V$ component of $T(V)$. 

Let $\pi_R:=\pi_{V^{r_1}\otimes\cdots \otimes V^{r_l}}$ and $\pi_{T(V^i)}$ for $i\in\{1,\dots,k\}$ denote the canonical projections
\begin{eqnarray*}
&&\pi_R:=\pi_{V^{r_1}\otimes\cdots \otimes V^{r_l}}:T(V)\to V^{\otimes R}\\
&&\pi_{T(V^i)}:T(V)\to T(V^i). 
\end{eqnarray*}
%

Given an element $v\in V$ and a multi-index $R=(r_1,\dots,r_l)$,  we introduce the element $v_R$ as follows:
\begin{equation*}
v_R:=(\pi_{(r_1)}v)\otimes\cdots\otimes(\pi_{(r_l)}v)\in V^{\otimes R}.
\end{equation*}

The set $\BPi{s}$ defined by
\begin{equation*}
\BPi{s}:=\left\{a\in T(V) | \ \forall R\in \APi{s}, \ \pi_R(a)=0 \right\}
\end{equation*}
is an ideal in $T(V)$.

The \emph{truncated tensor algebra of order} $(\Pi,s)$ is defined as the quotient algebra
\begin{equation*}
\TPi{s}(V):=T(V)/\BPi{s}.
\end{equation*}

\end{dfn}

We chose the tensor norms $\Vert\cdot\Vert_R$ for all
$R\in\mathcal{A}^k$ to satisfy 
\begin{equation*}
\Vert a\otimes b\Vert_{R\ast Q}\le \Vert a\Vert_R\Vert b \Vert_Q, \forall a\in V^{\otimes R}, \forall b\in V^{\otimes Q}.
\end{equation*}
We will drop the multi-index from the notation of the norm if it does not result in any ambiguity. 

\begin{dfn}[Control function]\label{_ch1_dfn_control}
Let $T$ be a positive real and $\Delta_T$ denote the set $\{(s,t)\in[0,T]\times[0,T]|s\le t\}$.
A \emph{control function}, or \emph{control}, on $[0,T]$ is a uniformly continuous non-negative function $\omega:\Delta_T\to[0,+\infty)$ which is super-additive, i.e.
\begin{equation*}
\omega(s,u)+\omega(u,t)\le\omega(s,t) \ \forall s,u,t\in[0,T], \ s\le u \le t
\end{equation*}
and for which $\omega(t,t)=0$ for all $t\in[0,T]$.
\end{dfn}

\begin{dfn}[Finite $\Pi$-variation] 
Let $\omega$ be a control on $[0,T]$. For a positive real $q$, the map
$\mfn{X}{}:\Delta_T\to\TPi{q}$ is \emph{multiplicative} if for all
$0\le s\le t \le T$, $\pi_{\eps}\mfn{X}{s,t}=1$ and for all $0\le s\le
t \le u\le T$,
\begin{equation*}
\mfn{X}{s,u}=\mfn{X}{s,t}\otimes\mfn{X}{t,u}.
\end{equation*}
Furthermore, $\mfn{X}{}$ is said to have finite $\Pi$ variation controlled by $\omega$ if there exists a positive $\beta$ such that
\begin{equation*}
\left\Vert \pi_R\left(\mfn{X}{s,t}\right)\right\Vert \le \frac{\omega(s,t)^{\degg_{\Pi}(R)}}{\beta^k\Gamma_{\Pi}(R)}
\end{equation*}
for all $(s,t)\in\Delta_T$ and for all $k$-multi-index $R\in\APi{q}$. 
\end{dfn}

The Extension theorem states that a $\Delta_T\to \TPi{1}(V)$ multiplicative functional of finite $\Pi$-variation can be uniquely extended to a $\Delta_T\to \TPi{q}(V)$ multiplicative functional of finite $\Pi$-variation for any positive $q$. This will allow us to define $\Pi$-rough paths as $\Delta_T\to \TPi{1}(V)$ multiplicative functionals satisfying certain properties (see Definition \ref{_dfn_PiRP}).

\begin{thm}[Extension theorem of multiplicative functionals of finite $\Pi$-variation]\label{_thm_PiExtension}
Let $\mfn{X}{}:\Delta_T\to \TPi{1}(V)$ be a multiplicative functional of finite $\Pi$-variation controlled by $\omega$. Then for every $k$-multi-index $R\in\mathcal{A}^{k}\setminus \APi{1}$, there exists a unique continuous function $\mfnc{X}{R}{}:\Delta_t\to V^{\otimes R}$ such that 
\begin{equation*}
(s,t)\mapsto \mfn{X}{s,t}=\sum_{R\in\mathcal{A}^{k}}\mfnc{X}{R}{s,t}\in T(V)
\end{equation*}
is a multiplicative functional of finite $\Pi$-variation controlled by $\omega$ in the following sense:
\begin{equation*}
\left\Vert \mfnc{X}{R}{s,t}\right\Vert\le \frac{\omega(s,t)^{\frac{n_1(R)}{p_1}+\cdots+\frac{n_k(R)}{p_k}}}{\beta^k\left(\frac{n_1(R)}{p_1}\right)!\cdots\left(\frac{n_k(R)}{p_k}\right)!}
=
\frac{\omega(s,t)^{\degg_{\Pi}(R)}}{\beta^k\Gamma_{\Pi}(R)}
\end{equation*} 
for all $R\in\mathcal{A}^k$, where
\begin{equation*}
\beta\ge \left(p_1^2\cdots p_k^2\left(1+\sum_{r=3}^{\infty}\left(\frac{2}{r-2}\right)^{s_{m^*+1}}\right)\right)^{1/k}
\label{_ch2_eq_ConditionBeta}
\end{equation*}
and $s_{m^*}$ and $s_{m^*+1}$ are the unique pair of adjacent elements of the ordered $S^{\Pi}$ for which $s_{m^*}\le 1 < s_{m^*+1}$.
\end{thm} 
A proof of this theorem based on the proof of the extension theorem of $p$-rough paths (ref. \cite{rp_paper}) is derived in \cite{lggy_thesis}. 

\begin{dfn}[$\Pi$-rough paths]\label{_dfn_PiRP}
A \emph{$\Pi$-rough path} in $V$ is a continuous $\Delta_T\to \TPi{1}(V)$ multiplicative functional $\mfn{X}{}$ with finite $\Pi$-variation controlled by some control $\omega$. 

The space of $\Pi$-rough paths is denoted by $\Omega_{\Pi}(V)$. 
\end{dfn}
\begin{dfn}\label{_ch2_dfn_p_var_metric}
Let $C_{0,\Pi}\left(\Delta_T,\TPi{1}(V)\right)$ denote the space of all continuous functions from the simplex $\Delta_T$ into the truncated tensor algebra $\TPi{1}(V)$ with finite $\Pi$-variation. The $\Pi$-variation metric $d_{\Pi-\text{\protect\textup{var}}}$ on this linear space if defined as follows
\begin{equation*} 
d_{\Pi-\text{\protect\textup{var}}}(\mfn{X}{},\mfn{Y}{}):=\max_{R\in\APi{1}}\sup_{\mathcal{D}\in\mathcal{P}([0,T])}\left(
\sum_{\mathcal{D}}\left\Vert
\pi_R\left(\mfn{X}{t_{l-1},t_l}-\mfn{Y}{t_{l-1},t_l}\right)
\right\Vert^{1/\degg_{\Pi}(R)}
\right)^{\degg_{\Pi}(R)}
\end{equation*}
%
\end{dfn}

The following subset of the space of $\Pi$-rough paths is crucial for our further analysis.

\begin{dfn}[Geometric $\Pi$-rough path] A \emph{geometric $\Pi$-rough
    path} is a $\Pi$-rough path which can be expressed as a limit of
  $(1)$-rough paths\footnote{Here $(1)$ denotes the $1$-tuple with the
  single element $1$.} (or \emph{smooth} rough paths) in the $\Pi$-variation distance. The space of geometric $\Pi$-rough paths in $V$ is denoted by $G\Omega_{\Pi}(V)$.
\end{dfn}

\begin{rem}
In the special case, when $k=1$ and $\Pi=(p)$ for some $p\ge 1$, we will use the simplified notation: "finite $p$-variation", "$p$-rough paths", "$d_{p-\text{\protect\textup{var}}}$-distance" and "geometric $p$-rough paths". A direct definitions for these terms can be found in \cite{rp_paper} and \cite{rp_notes}. Furthermore for a $1$-multi index $R$ with length $j$, we will use the notation $\pi_j=\pi_R$, and we will write $T^{i}(V)$ for $\TPi{i/p}(V)$.
\end{rem}

\section{Integration with respect to $\Pi$-rough paths}\label{sec:Integration}
Lyons \cite{rp_paper} introduced integrals of \lipg one-forms along $p$-rough paths for $\gamma>p-1$. In this section, we introduce \lipGP one-forms (Definition \ref{_dfn_lipGP}) and integrals of \lipGP one-forms along $\Pi$ rough paths (Theorem \ref{_thm_lipGP_Integration}). 

First we define the $s$-symmetric maps for $s\in S^{\Pi}$.  

\begin{dfn}[$s$-symmetric maps] Let $i$ be a positive integer and
  $x\in V^{\otimes i}$. Let $x^{(i)}$ denote the $i$-symmetric part of
  $x$. Let $q\in(0,\infty]$, $X\in \TPi{q}(V)$ and $s\in S^{\Pi}$. Then the $s$-symmetric part of $X$ is defined as
\[
X^{(s)}:=\sum_{\degg_{\Pi}(R)=s}\pi_R\left(\sum_{i\in\mathbb{N}} \pi_i(X)^{(i)}\right). 
\]
A map $f$ defined on $\TPi{q}(V)$ is \emph{$s$-symmetric}, if for all $X\in\TPi{q}(V)$
\[
f(X)=f\left(X^{(s)}\right).
\]
\end{dfn}

Now, we can give the definition of \lipGP one-forms. 

\begin{dfn}[\lipGP one-forms]\label{_dfn_lipGP}
Let $\Pi=(p_1,\dots,p_k)$ and $\Gamma=(\gamma_1,\dots,\gamma_k)$ be
$k$-tuples with positive components, and let $F$ be a closed subset of $V$, $W$ be a Banach space. The function $\alpha:F\to L(V,W)$ is a $(\Pi,\Gamma)$-Lipschitz one-form on $F$ if $\alpha(u)=\sum_{i=1}^k\alpha_i(u)\circ\pi_{V_i}$ such that $\alpha_i:F\to L(V^i,W)$ and for each $i$ and $s_m<\gamma_i$ (where $s_m$ is the $m^{\text{th}}$ element in the ordered set $S^{\Pi}$) there exist functions $\alpha_i^{s_m}:F\to L\left(V^{(\Pi,s_m)},L(V^i,W)\right)$ taking values in the space of $s_m$-symmetric maps satisfying
\begin{align*}
\alpha_i^{s_m}(y)(v)=\sum_{s_m\le s_n<\gamma_i}\alpha_i^{s_n}(x)
	\left(
		v\otimes \sum_{\degg_{\Pi}(R)=s_n-s_m} \frac{(x-y)_R}{\Vert R\Vert!}
	\right)
	+R_i^{s_m}(x,y)(v)
\end{align*}
for all $x,y\in F$ and $v\in V^{\otimes(\Pi,s_m)}$, where $R_i^{s_m}:F\times F\to L\left(V^{(\Pi,s_m)},L(V^i,W)\right)$ with
\begin{align}
\Vert R_i^{s_m}(x,y) \Vert \le M\sum_{j=1}^k \Vert \pi_{V^j}(x-y)\Vert^{(\gamma_i-s_m)p_j}.\label{_eq_remainder_bound}
\end{align} 
\end{dfn}

In addition to the above definition and for practical reasons we introduce the functions $\alpha^{s_m}:F\to L\left(V^{(\Pi,s_m)},L(V,W)\right)$ for $s_m<\max_{1\le i\le k}\gamma_i=\gamma_{\max}$ defined by
\begin{equation*}
\alpha^{s_m}(v)(u)=\sum_{i, s_m<\gamma_i}\alpha^{s_m}_i(v)(u)\circ\pi_{V^i}, \  \forall v\in F, \forall u\in V^{(\Pi,s_m)}.
\end{equation*}
Note that $\alpha^{s_m}$ takes $s_m$-symmetric linear maps as values. Furthermore we introduce the functions $R_{s_m}:F\times F \to L\left(V^{\otimes(\Pi,s_m)},L(V,W)\right)$ defined by
\begin{equation}
R^{s_m}(x,y)(u)=\sum_{i, s_m<\gamma_i}R^{s_m}_i(x,y)(u)\circ\pi_{V^i}, \ \forall x,y\in F, \forall u\in V^{(\Pi,s_m)}.\label{_dfn_LipPiGammaR}
\end{equation}

\begin{rem}
Note that for $k=1$, $\Pi=(p)$ and $\Gamma=(\hat{\gamma})$ Definition \ref{_dfn_lipGP} simplifies to the classical definition of \lipg functions (ref. \cite{rp_paper}) for $\gamma=\hat{\gamma}p$, although the notation is slightly different. 
\end{rem}

\begin{rem}
The above definition can be interpreted as a "decomposition and partial reconstruction" of the classical \lipg functions. In the classical definition $f^j$ (i.e. the $j^{\text{th}}$ term in the expansion) is an $F\to L\left( V^{\otimes j}, L(V,W)\right)$-valued function, which can be decomposed as 
\begin{align*}
f^j(x)=\sum_{\Vert R\Vert=j} f^{j,R}(x)\circ \pi_{V^{\otimes R}}
\end{align*}
for $x\in F$, where $f^{j,R}(x)$ is a linear function in $L(V^{\otimes R},L(V,W))$. In the above definition these $f^{j,R}$ functions are grouped and summed by the degree of $R$ 
leaving out those that have degree corresponding to multi-index with degree greater than a certain value ($\gamma_i$). Similar decomposition and partial reconstruction is done with the remainder terms. Note that the condition \eqref{_eq_remainder_bound} on the remainder term are weaker compared to the homogeneous case ($k=1$, $\Pi=(p_{\max}), \Gamma=(\hat{\gamma})$). 
\end{rem}

\begin{rem}
In the special case when all the multi-indices $R$ of degree less than $\gamma_i$ are of the form $(j,\dots,j)$ for some $1\le j\le k$ (for example $k=2$ and $\frac{1}{p_1}+\frac{1}{p_2}>\gamma_j$ for $j=1$ or $2$) the term
\[
v\otimes \sum_{\degg_{\Pi}(R)=s_n-s_m} \frac{(x-y)_R}{\Vert R\Vert!}
\]
for $v\in V^{\otimes(\Pi,s_m)}$ is actually homogeneous, i.e. lies in $(V^j)^{\otimes n}$ where $n=s_m p_j$. In this case the condition (in the above definition) on $\alpha_i$ is equivalent to the following. For each $j=1,\dots,k$ the function 
\[
 x\mapsto\alpha_i((y^1,\dots,y^{j-1},x,y^{j+1},\dots,y^k)), \ x\in V^j
\]
for fixed
 \[
 (y^1,\dots,y^{j-1},j^{j+1},\dots,j^k)\in V^1\oplus\cdots\oplus V^{j-1}\oplus V^{j+1}\oplus\cdots V^k 
\] 
is $\text{Lip}^{\gamma_ip_j}$ (in the classical sense) with Lipschitz norm uniform in $(y^1,\dots,y^{j-1},j^{j+1},\dots,j^k)$. 

In the case, when $\frac{1}{p_j}>\gamma_i$, the $\text{Lip}^{\gamma_ip_j}$ condition is equivalent to $\gamma_ip_j$-H\"older continuity.  
\end{rem}

In Theorem \ref{_thm_lipGP_Integration}, an integral approximation
formula is introduced and we prove the existence of a unique rough path associated with the integral approximating formula. This unique rough path is referred to as the integral of a one-form along a $\Pi$-rough path. To prove the existence and uniqueness of the rough path associated with the integral approximating formula we reformulate the problem in terms of almost $p$-rough paths.

\begin{dfn}[Almost $p$-rough path]\label{_dfn_AlmostpRP} Let $p\ge 1$ be a real number and $\omega$ a control. A function $\mfn{Y}{}:\Delta_T\to T^{(\lfloor p \rfloor)}(V)$ as an \emph{almost $p$-rough path} if 
\begin{enumerate}[(i)]
\item $\mfn{Y}{}$ has finite $p$-variation controlled by $\omega$, i.e.
\begin{equation*}
\left\Vert \pi_i\left(\mfn{Y}{s,t}\right)\right\Vert\le\frac{\omega(s,t)^{\frac{i}{p}}}{\beta\left(\frac{i}{p}\right)!}
\ \forall i=1,\dots,\lfloor p\rfloor, \ \forall(s,t)\in\Delta_T
\end{equation*}
\item $\mfn{Y}{}$ is \emph{almost multiplicative} in the sense
\begin{equation*}
\left\Vert \pi_i\left(\mfn{Y}{s,u}\otimes\mfn{Y}{u,t}-\mfn{Y}{s,t}\right)\right\Vert\le\omega(s,t)^{\theta} \ \forall i=1,\dots,\lfloor p\rfloor, \ \forall s,u,t \in[0,T], \ s\le u\le t
\end{equation*}
and for some $\theta>1$.
\end{enumerate}
\end{dfn}

Almost rough paths have the crucial property that each one of them determines a rough path in the sense of Theorem \ref{_thm_RPtoAlmostRP}. This property is exploited when we derive the existence and uniqueness of integrals along $\Pi$-rough paths. 

\begin{thm}\label{_thm_RPtoAlmostRP}
Let $p\ge 1$ be a real number and $\omega$ be a control. Let $\mfn{Y}{}:\Delta_T\to T^{(\lfloor p\rfloor)}(V)$ be an almost $p$-rough path with $p$-variation controlled by $\omega$ as in Definition \ref{_dfn_AlmostpRP}. 
Then there exists a unique $p$-rough path $\mfn{X}{}:\Delta_T\to T^{(\lfloor p\rfloor)}(V)$ such that
\begin{equation*}
\sup_{
\begin{smallmatrix}
0\le s<t\le T\\
i=0,\dots,\lfloor p \rfloor
\end{smallmatrix}
}
\frac{\left\Vert \pi_i\left(\mfn{X}{s,t}-\mfn{Y}{s,t}\right)\right\Vert}{\omega(s,t)^{\theta}}< +\infty.
\end{equation*}
Moreover, there exists a constant $K$ depending only on $p$,$\theta$ and $\omega(0,T)$, such that the supremum is smaller than $K$, and the $p$-variation of $\mfn{X}{}$ is controlled by $K\omega$. 
\end{thm}
The reader is referred to \cite{rp_paper} and \cite{rp_book} for proof. 

Although it is not required for the main result of this section, Theorem \ref{_thm_RPtoAlmostRP} can be extended for general $k>1$ and $k$-tuple $\Pi$ as follows. 

\begin{thm}\label{_ch2_thm_AlmostPiRpExt}
Let the Banach space $V$ be of the form $V=V^1\oplus\cdots\oplus V^k$ for some Banach spaces $V^1,\dots, V^k$.
Let $\Pi=(p_1,\dots,p_k)$ denote a $k$-tuple as in Definition \ref{_dfn_PiBasics} and let $\omega$ be a control. Let 
the functional $\mfn{Y}{}:\Delta_T\to \TPi{1}(V)$ be a $\theta$-almost $\Pi$-rough path controlled by $\omega$, i.e.

\begin{enumerate}[(i)]
\item it has a finite $\Pi$-variation controlled by $\omega$:
\begin{equation*}
\left\Vert \mfnc{X}{R}{s,t}\right\Vert \le \frac{\omega(s,t)^{\degg_{\Pi}(R)}}{\beta^k\Gamma_{\Pi}(R)}
\end{equation*}
for all $(s,t)\in\Delta_T$ and for all multi-index $R\in\APi{1}$. 
\item it is \emph{almost-multiplicative}, i.e. there exists $\theta>1$ such that
\begin{equation*}
\left\Vert\pi_R\left(\mfn{X}{s,u}\otimes\mfn{X}{u,t}-\mfn{X}{s,t}\right)\right\Vert\le\omega(s,t)^{\theta} \ \forall s<u<t\in[0,T], \ \forall R\in\APi{1}.
\end{equation*}
\end{enumerate} 

Then there exists a unique $\Pi$-rough path $\mfn{X}{}:\Delta_T\to\TPi{1}(V)$ such that 
\begin{equation}
\sup_{
\begin{smallmatrix}
0\le s<t\le T\\
R\in\APi{1}
\end{smallmatrix}
}
\frac{
\left\Vert
\pi_R\left(\mfn{X}{s,t}-\mfn{Y}{s,t}\right)
\right\Vert
}{
\omega(s,t)^{\theta}
}<+\infty.\label{_eq_AlmostRPSup}
\end{equation}
Moreover, there exists a constant $K$ which depends only on $\Pi$, $\theta$ and $\omega(0,T)$ such that the supremum \eqref{_eq_AlmostRPSup} is smaller than $K$ and the $\Pi$-variation of $X$ is controlled by $K\omega$.
\end{thm}

The proof of Theorem \ref{_ch2_thm_AlmostPiRpExt} is sketched in \cite{lggy_thesis} and based on the proof of Theorem \ref{_thm_RPtoAlmostRP} as derived in \cite{rp_notes}.

Finally, we can state the main theorem of the section. 

\begin{thm}[Integration of \lipGP one-forms]\label{_thm_lipGP_Integration}
Let $V$ and $W$ be Banach spaces, such that $V=V^1\oplus\cdots\oplus V^k$ for some Banach spaces $V^1,\dots, V^k$.
Let $\Pi=(p_1,\dots,p_k)$ denote a $k$-tuple as in Definition \ref{_dfn_PiBasics} with $p_{\max}=\max_{1\le i\le k}p_i$ and let $\omega$ be a control.
Let $\mfn{Z}{}:\Delta_T\to\TPi{1}(V)$ be a geometric $\Pi$-rough path controlled by $\omega$.
 Let $\Gamma=(\gamma_1,\dots,\gamma_k)$ be a real $k$-tuple such that $\gamma_i>1-1/p_i$ for $i=1,\dots,k$ and $\gamma_{\max}=\max_{1\le i\le k}\gamma_i$ . Finally let  $\alpha:V\to L(V,W)$ be a \lipGP function as in Definition \ref{_dfn_lipGP}. 

Then $\mfn{Y}{}:\Delta_T\to T^{((p_{\max}),1)}(W)$ defined for all $(s,t)\in\Delta_T$ 
by

\begin{eqnarray}
&&\mfnc{Y}{n}{s,t}:=\pi_{W^{\otimes n}}\left(\mfn{Y}{s,t}\right)
=\nonumber\\
&& 
\sum_{
s_{m_1}+\cdots+s_{m_n}<\gamma_{\max}
} 
\alpha^{s_{m_1}}\left(\pi_V\left(\mfn{Z}{0,s}\right)\right)\otimes\cdots\otimes
\alpha^{s_{m_n}}\left(\pi_V\left(\mfn{Z}{0,s}\right)\right)
\sum_{
\begin{smallmatrix}
R_1,\dots,R_n\in\mathcal{A}^k\\
\degg_{\Pi}(R_i-)=s_i, \ i=1,\dots, n \\
\sigma\in OS( \Vert R_1\Vert,\dots, \Vert R_n\Vert )
\end{smallmatrix}}
\sigma^{-1}\pi_{R_1\ast\cdots\ast R_n}\left(\mfn{Z}{s,t}\right)\nonumber\\
&& \label{_eq_IntegralApproxSum}
\end{eqnarray}
is an almost 
$p_{\max}$-rough path,
where $OS( k_1,\dots, k_n )$ denotes the subset
of the symmetric group $S_K$ for $K=k_1+\cdots+k_n$, such that
for all $\sigma\in OS(k_1,\dots, k_n )$, we
have
\[
\sigma(1)<\sigma(2)<\cdots<\sigma(k_1), \
\sigma(k_1+1)<\cdots<\sigma(k_1+k_2),\ \dots,
\]
\[
\sigma(K-k_n+1)<\cdots<\sigma(K), \text{ and } \sigma(k_1)<\sigma(k_2)<\cdots<\sigma(k_n),
\]
moreover with a slight abuse of notation for $x_1\otimes \cdots\otimes
x_K\in V^{R_1\ast\cdots\ast R_n}$ and $\sigma\in OS( \Vert
R_1\Vert,\dots, \Vert R_n\Vert )$ we define
\[
\sigma (x_1\otimes \cdots\otimes x_K) = x_{\sigma(1)}\otimes\cdots\otimes x_{\sigma(K)}.
\]
\end{thm}

Theorem \ref{_thm_lipGP_Integration} leads to the following definition. 
\begin{dfn}[Integration of \lipGP one-forms]\label{_dfn_lipGP_Integration}
Let $V$ and $W$ be Banach spaces, such that $V=V^1\oplus\cdots\oplus V^k$ for some Banach spaces $V^1,\dots, V^k$.
Let $\Pi=(p_1,\dots,p_k)$ denote a $k$-tuple as in Definition \ref{_dfn_PiBasics} with $p_{\max}=\max_{1\le i\le k}p_i$ and let $\omega$ be a control.
Let $\mfn{Z}{}:\Delta_T\to\TPi{1}(V)$ be a geometric $\Pi$-rough path controlled by $\omega$.
 Let $\Gamma=(\gamma_1,\dots,\gamma_k)$ be a real $k$-tuple such that $\gamma_i>1-1/p_i$ for $i=1,\dots,k$. Finally let  $\alpha:V\to L(V,W)$ be a \lipGP function.

Let $\mfn{Y}{}:\Delta_T\to T^{((p_{\max}),1)}(W)$ be the almost $p_{\max}$-rough path defined by Theorem \ref{_thm_lipGP_Integration}. The unique $(p_{\max})$-rough path associated to $\mfn{Y}{}$ by Theorem \ref{_thm_RPtoAlmostRP} is called the integral of $\alpha$ along $\mfn{Z}{}$ and it is denoted by
\begin{equation*}
\int_{\cdot}^{\cdot}\alpha(\mfn{Z}{})\dd\mfn{Z}{}:\Delta_T\to T^{((p_{\max}),1)}(W).
\end{equation*}
\end{dfn}

\begin{rem}
In the general case the integral is a $p_{\max}$-rough path in the sense of \cite{rp_paper}. However, for special forms of the \lipGP one-form $\alpha$, the integral itself is a $\Pi$-rough paths. The reader is referred to \cite{lggy_thesis} for examples. 
\end{rem}

In the remaining part of the section, we present a proof of Theorem \ref{_thm_lipGP_Integration}. 

Equation \eqref{_eq_IntegralApproxSum} describes an integral approximating formula projected on $W^{\otimes n}$. 
The intuition behind this formula comes from integrals with respect to paths of finite length. 
In particular, let $Z:[0,T]\to V$ be a path of finite variation. For a multi-index $R=(r_1,\dots,r_l)$, let $\mfnc{Z}{R}{s,t}\in V^{\otimes R}$ be defined as
\begin{equation}
\mfnc{Z}{R}{s,t}=\int_{s<u_1<\cdots<u_l<t} d\pi_{r_1}(Z_{u_1})\otimes\cdots\otimes d\pi_{r_l}(Z_{u_l}).
\label{_eq_iterated_int}
\end{equation}
Furthermore, let the function $\mfnc{Y}{1}{}:\Delta_T\to W$ be
defined for all $(s,t)\in\Delta_T$ by

\begin{equation}\label{_eq_ch2_1rp_intsum}
\mfnc{Y}{1}{s,t}:=\sum_{
s_{m}<\gamma_{\max}
} 
\alpha^{s_{m}}\left(\mfn{Z}{s}\right)
\sum_{
\begin{smallmatrix}
R\in\mathcal{A}^k\\
\degg_{\Pi}(R-)=s_{m}
\end{smallmatrix}
}
\mfnc{Z}{R}{s,t}=\int_s^t\alpha(Z_u)dZ_u- \int_s^tR^0(Z_s,Z_u)dZ_u. 
\end{equation}
Then 
\begin{align}
\mfnc{Y}{n}{s,t}&:=\int_{s<u_1<\cdots<u_n<t}d \mfnc{Y}{1}{s,u_1}\otimes\cdots\otimes d\mfnc{Y}{1}{s,u_n}
\nonumber\\
&
=\int_{s<u_1<\cdots<u_n<t}
\sum_{
s_{m_1}<\gamma_{\max}
} 
\alpha^{s_{m_1}}\left(\mfn{Z}{s}\right)
\sum_{
\begin{smallmatrix}
R\in\mathcal{A}^k\\
\degg_{\Pi}(R-)=s_{m_1}
\end{smallmatrix}
}
d\mfnc{Z}{R}{s,u_1}\otimes\nonumber\\
&
\ \ \ \ \ \  \ \ \ \ \ \ \ \ \ \ \ \ \ \ \ \ \ \ \ \ \ \ \ \ \ \ \ \ \ \ \ \ \ \ \ \ \ \ 
\cdots\otimes\sum_{
s_{m_n}<\gamma_{\max}
} 
\alpha^{s_{m_n}}\left(\mfn{Z}{s}\right)
\sum_{
\begin{smallmatrix}
R\in\mathcal{A}^k\\
\degg_{\Pi}(R-)=s_{m_n}
\end{smallmatrix}
}
d\mfnc{Z}{R}{s,u_n}\nonumber\\
&=\sum_{s_{m_1},\cdots,s_{m_n}<\gamma_{\max}}\alpha^{s_{m_1}}\left(\mfn{Z}{s}\right)
\otimes\cdots\otimes\alpha^{s_{m_n}}\left(\mfn{Z}{s}\right)
\int_{s<u_1<\cdots<u_n<t}d\mfnc{Z}{R}{s,u_1}\otimes\cdots\otimes d\mfnc{Z}{R}{s,u_n}
\nonumber\\
&=\sum_{
s_{m_1},\cdots,s_{m_n}<\gamma_{\max}
} 
\alpha^{s_{m_1}}\left(\mfn{Z}{s}\right)\otimes\cdots\otimes
\alpha^{s_{m_n}}\left(\mfn{Z}{s}\right)
\sum_{
\begin{smallmatrix}
R_1,\dots,R_n\in\mathcal{A}^k\\
\degg_{\Pi}(R_i-)=s_i, \ i=1,\dots,n\\
\sigma\in OS(\Vert R_1\Vert,\dots,\Vert R_n\Vert)
\end{smallmatrix}
}
\sigma^{-1}\mfnc{Z}{R_1\ast\cdots\ast R_n}{s,t} \label{_eq_PiIntLem2}
\end{align}
for $n=2,\dots,\lfloor p_{\max}\rfloor$.

Equation \eqref{_eq_PiIntLem2} is an adaptation of the results of Section 4.2. of \cite{rp_notes}. 

\begin{lem}\label{_lem_PiIntLem1}
If $\mfn{Z}{}:[0,T]\to V$ is path of finite variation and piece-wise differentiable, then for any $s<u<t$ in $[0,T]$,
\begin{eqnarray}
&&
\sum_{s_m<\gamma_{\max}}\alpha^{s_m}
\left(
	\mfn{Z}{s}
\right)
\left(
	\sum_{\degg_{\Pi}(R)=s_m} \mfnc{Z}{R}{s,t}
\right)
\left(
	\mfn{\dot{Z}}{t}
\right)=\nonumber\\
&& \ \ \ \ \ \  \ \ \ \ \ \ \ \ \ 
\sum_{s_m<\gamma_{\max}}
\Big(
\alpha^{s_m}
\left(
	\mfn{Z}{u}
\right)
-R_{s_m}
\left(
	\mfn{Z}{s},
	\mfn{Z}{u}
\right)
\Big)
\left(
	\sum_{\degg_{\Pi}(R)=s_m}\mfnc{Z}{R}{u,t}
\right)
\left(
	\mfn{\dot{Z}}{t}
\right)\nonumber\\
&& \label{_eq_PiIntLem1}
\end{eqnarray}
\end{lem}

The proof of the lemma is analogous to the proof of Lemma 5.5.2 in \cite{rp_book}.

\begin{lem}\label{_lem_PiIntLem3}
Let $Z:[0,T]\to V$ be a path of finite variation. Let the map
\[
Y=(1,Y^1,\dots,Y^{\lfloor p_{\max}\rfloor}):\Delta_T\to \Rr\oplus W\oplus W^{\otimes 2} 
\otimes\cdots\otimes W^{\otimes \lfloor p_{\max}\rfloor} 
\]
be defined by equations \eqref{_eq_ch2_1rp_intsum} and \eqref{_eq_PiIntLem2}.

 Then for all $s<u<t$ in $[0,T]$,
\begin{equation}
\mfn{Y}{s,u}\otimes\mfn{Y}{u,t}-\mfn{Y}{s,t}=\mfn{Y}{s,u}\otimes\mfn{N}{s,u,t}\label{_eq_PiIntLem3}
\end{equation}
where 
\begin{eqnarray}
&& \mfnc{N}{i}{s,u,t}=\pi_{W^{\otimes i}}\mfn{N}{s,u,t}:=\nonumber\\
&&
\sum_{
\begin{smallmatrix}
s_{m_1},\dots,s_{m_i}<\gamma_{\max}\\
\eps_1,\dots,\eps_i\in\{0,1\}\\
\eps_1\cdots\eps_i=0
\end{smallmatrix}
}
\beta^{\eps_1}_{s_{m_1}}\left(
	\mfn{Z}{s},\mfn{Z}{u}
\right)
\cdots
\beta^{\eps_i}_{s_{m_i}}\left(
	\mfn{Z}{s},\mfn{Z}{u}
\right)
\sum_{
\begin{smallmatrix}
R_1,\dots,R_i\in\mathcal{A}^k\\
\degg_{\Pi}(R_j-)=s_j, \ j=1,\dots,i\\
\sigma\in OS(\Vert R_1\Vert,\dots,\Vert R_i\Vert)
\end{smallmatrix}
}
\sigma^{-1}\mfnc{Z}{R_1\ast\cdots\ast R_i}{s,t}\nonumber\\
&& \label{_eq_PiIntLem3_2}
\end{eqnarray}
with
\begin{equation*}
\beta^{\eps}_{s_{m}}\left(
	\mfn{Z}{s},\mfn{Z}{u}
\right)=\left\{
\begin{array}{rl}
R_{s_{m}}\left(
	\mfn{Z}{s},\mfn{Z}{u}
\right) & \text{if } \eps=0,\\
-\alpha^{s_{m}}\left(
	\mfn{Z}{s}\right) & \text{if } \eps=1.
\end{array}\right.
\end{equation*}
\end{lem}

The proof is based on Lemma \ref{_lem_PiIntLem1} and the equation \eqref{_eq_PiIntLem2}, and is analogous to the proof of Lemma 5.5.3 in \cite{rp_book}.

\begin{rem}\label{_lem_PiIntRem2}
The equation \eqref{_eq_PiIntLem2} and Lemmas \ref{_lem_PiIntLem1} and \ref{_lem_PiIntLem3} are stated for a smooth rough path $\mfn{Z}{}$. However for each of the equalities \eqref{_eq_PiIntLem1}, \eqref{_eq_PiIntLem2} and \eqref{_eq_PiIntLem3}, both the right-hand side and the left-hand side are continuous in the $\Pi$-variation metric. This fact extends the lemmas for geometric $\Pi$-rough paths and this is the key to the next proof. 
\end{rem}

We now prove Theorem \ref{_thm_lipGP_Integration}.

\begin{proof} [Proof of Theorem \ref{_thm_lipGP_Integration}:]
First we prove that $\mfnh{Y}{}:\Delta_T\to T^{((p_{\max}),1)}(W)$, defined by
\begin{eqnarray*}
&&\mfnch{Y}{n}{s,t}= \\
&& 
\sum_{
s_{m_1},\cdots,s_{m_n}<\gamma_{\max}
} 
\alpha^{s_{m_1}}\left(\pi_V\left(\mfn{Z}{0,s}\right)\right)\otimes\cdots\otimes
\alpha^{s_{m_n}}\left(\pi_V\left(\mfn{Z}{0,s}\right)\right)
\sum_{
\begin{smallmatrix}
R_1,\dots,R_n\in\mathcal{A}^k\\
\degg_{\Pi}(R_i-)=s_i, \ i=1,\dots, n \\
\sigma\in OS(\Vert R_1\Vert,\dots,\Vert R_n\Vert)
\end{smallmatrix}}
\sigma^{-1}\pi_{R_1\ast\cdots\ast R_n}\left(\mfn{Z}{s,t}\right)
\end{eqnarray*}
is an almost $p_{\max}$-rough path. 
Each term in the above sum is of the form 
\begin{equation}
\alpha^{s_{m_1}}\left(\pi_V\left(\mfn{Z}{0,s}\right)\right)\otimes\cdots\otimes
\alpha^{s_{m_n}}\left(\pi_V\left(\mfn{Z}{0,s}\right)\right)
\sigma^{-1}\pi_{R_1\ast\cdots\ast R_n}\left(\mfn{Z}{s,t}\right)\label{_eq_ch2_BoundOnTerm}
\end{equation}
where $\degg_{\Pi}(R_i-)=s_{m_i}$. Since such a term is bounded by $C_0\Vert\alpha\Vert^{n}_{\text{Lip}^{\Gamma,\Pi}}\omega(s,t)^{n/p_{\max}}$ where $C_0$ only depends on $\Gamma$, $\Pi$ and $\omega(0,T)$, this implies that condition $i)$ of Definition \ref{_dfn_AlmostpRP} is satisfied.

We prove condition $ii)$ by giving a bound on the norm of 
\begin{equation*}
(\mfnh{Y}{s,u}\otimes\mfnh{Y}{u,t})^n-\mfnch{Y}{n}{s,t}=\sum_{i=0}^n\mfnch{Y}{i}{s,u}\otimes\mfnc{N}{n-i}{s,u,t}.
\end{equation*}

The representation of $\mfnc{N}{n-i}{s,u,t}$ in the equation \eqref{_eq_PiIntLem3_2} implies that 
there is at least one factor of the form 
$R^{s_m}\left(\pi_V\left(\mfn{Z}{0,s}\right),\pi_V\left(\mfn{Z}{0,u}\right)\right)$.
 Considering the representation \eqref{_dfn_LipPiGammaR} of $R^{s_m}$ and the error bound \eqref{_eq_remainder_bound} on $R_i^{s_m}$, the following bound is implied:
\begin{equation*}
\left\Vert R^{s_m}\left(\pi_V\left(\mfn{Z}{0,s}\right),\pi_V\left(\mfn{Z}{0,u}\right)\right)\right\Vert
\le
M\sum_{i,s_m<\gamma_i}\sum_{j=1}^k\left\Vert\pi_{V_j}\left(\mfn{Z}{s,u}\right)\right\Vert^{(\gamma_i-s_m)p_j}
\le
M\sum_{i,s_m<\gamma_i}\sum_{j=1}^k\omega(s,t)^{\gamma_i-s_m}.
\end{equation*}
Moreover considering that $R^{s_m}_i(x,y)(u)$ only acts on elements of $V^i$,
there exists a constant $C_1$ depending only on $\Vert\alpha\Vert_{\text{Lip}^{\Gamma,\Pi}}$, $\Gamma$, $\Pi$ and $\omega(0,T)$ such that 
\begin{equation*}
\Vert(\mfnh{Y}{s,u}\otimes\mfnh{Y}{u,t})^n-\mfnch{Y}{n}{s,t}\Vert\le C_1
\sum_{i=1}^k\omega(s,t)^{\gamma_i+(1/p_i)}.
\end{equation*}
By the choice of $\Gamma$, $\theta:=\min_{1\le i\le k}(\gamma_i+(1/p_i))\ge 1$, which implies that 
there exists a constant $C$ depending only on $\Vert\alpha\Vert_{\text{Lip}^{\Gamma,\Pi}}$, $\Gamma$, $\Pi$ and $\omega(0,T)$
such that 
\begin{equation*}
\Vert(\mfnh{Y}{s,u}\otimes\mfnh{Y}{u,t})^n-\mfnch{Y}{n}{s,t}\Vert\le C
\omega(s,t)^{\theta}
\end{equation*}
and hence $\mfn{Y}{}$ is a $\theta$-almost $p_{\max}$-rough path. 

Arguments analogous to Proposition 4.10 in \cite{rp_notes} prove that $\mfn{Y}{}$ is also a $\theta$-almost $p_{\max}$-rough path and furthermore that the $p_{\max}$-rough associated to $\mfn{Y}{}$ by Theorem \ref {_thm_RPtoAlmostRP} coincides with the $p_{\max}$-rough path associated to $\mfnh{Y}{}$. 
\end{proof}

\begin{thm}\label{_thm_BoundOnInt}
Under the conditions of Definition \ref{_dfn_lipGP_Integration}, there exists a constant $K$ depending only on 
 $\Gamma$, $\Pi$ and $\omega(0,T)$, such that
\begin{equation*}
\left\Vert\pi_{W^{\otimes i}}\left(\int_{s}^{t}\alpha(\mfn{Z}{})\dd\mfn{Z}{}\right)\right\Vert\le K
\Vert\alpha\Vert^{i}_{\text{Lip}(\Pi,\Gamma)}
\omega(s,t)^{\frac{i}{p_{\max}}}.
\end{equation*}
\end{thm}

The proof is analogous to the proof of Theorem 4.12. of \cite{rp_notes}

\section{Differential equations driven by $\Pi$-rough paths}\label{sec:RDE}
When stating and proving the slightly generalised version of Lyons' Universal Limit Theorem, we will refer to (linear) images of $\Pi$-rough paths in the following sense.
\begin{dfn}[Image by a function] Let $Z:\Delta_T\to\TPi{1}(V)$ be a geometric $\Pi$-rough path as in section \ref{sec:PiRPs}. Let $f:V\to W$ be a \lipGP function for some $k$-tuple $\Gamma=(\gamma_1,\dots,\gamma_k)$ satisfying 
 $\gamma_i>1-1/p_i$ for $i=1,\dots,k$. Then the integral  $\int df(Z)dZ$ is by definition a rough path in $\Omega_{(p_{\max})}(W)$. We will denote this rough path by $\hat{f}(Z)$. 
 \end{dfn}
We make use of linear images of rough paths and in particular projections of rough paths. E.g. if $X$ is a rough path in $\Omega_{\Pi}(V)$ then the image of $X$ under the projection $\pi_{V^i}$ will be denoted by $\hat{\pi}_{V^i}(X)$. 

Now we can formally introduce differential equations driven by geometric $\Pi$-rough paths. 

\begin{dfn}[Differential equations driven by $\Pi$-rough paths]\label{_dfn_RDE}
 Let $k\ge 1$ be an integer, $V$ and $W$ Banach spaces, such that $V=V^1\oplus\cdots\oplus V^k$ for some Banach spaces $V^1,\dots, V^k$. 
Let $\Pi=(p_1,\dots,p_k)$ denote a $k$-tuple and  $\Pi^*=(p_1,\dots,p_k,p_{\max})$ denote a $(k+1)$-tuple both as in Definition \ref{_dfn_PiBasics}. 
Let $f:V\oplus W\to L(V,W)$ be a 
function. Finally let  $\mfn{X}{}\in G\Omega_{\Pi}(V)$ be a geometric $\Pi$-rough path and $\xi$ an element in $W$. 

We will say that $\mfn{Z}{}\in G\Omega_{\Pi^*}(V\oplus W)$  is a solution of the differential equation 
\begin{equation}
\dd\mfn{Y}{t}=f(\mfn{X}{t},\mfn{Y}{t})\dd \mfn{X}{t}, \ \mfn{Y}{0}=\xi\label{_eq_dfn_RDE}
\end{equation}
%
if $\hat{\pi}_{V}(\mfn{Z}{})=\mfn{X}{}$ and
\begin{equation}\label{_eq_dfn_rde}
\mfn{Z}{}=\int h_0(\mfn{Z}{})\dd\mfn{Z}{}
\end{equation}
where 
$h_0:V\oplus W\to\text{\protect{\textup{End}}}(V\oplus W)$ is defined by
\begin{equation*}
h_0(x,y)=\left(\begin{array}{cc}
\text{Id}_V & 0 \\
f(x,y+\xi) & 0
\end{array}\right)
\end{equation*}
provided the integral \eqref{_eq_dfn_rde} is well defined.
\end{dfn}

In the remainder of the section we give a sufficient condition for the existence and uniqueness of solution to the equation \eqref{_eq_dfn_RDE}. We will assume the existence of the function $g_{\xi}:V\times W\times W\to L(W,L(V,W))$ such that
\[
f(x,y_1+\xi)-f(x,y_2+\xi)=g_{\xi}(x,y_1,y_2)(y_1-y_2), \text{ for all } x\in V, \ y_1,y_2\in W. 
\]
We introduce the one-forms $h_1:V\oplus W \oplus W\to \text{\protect{\textup{End}}}(V\oplus W\oplus W)$ and $h_2:V\oplus W \oplus W \oplus W\to \text{\protect{\textup{End}}}(V\oplus W\oplus W \oplus W)$ as follows:
\begin{align*}
h_1(x,y_1,y_2)=\left(
\begin{array}{ccc}
Id_V & 0 & 0 \\
0 & 0 & Id_W \\
f(x,y_2+\xi) & 0 & 0
\end{array}
\right)
\end{align*}
\begin{align*}
h_2(x,y_1,y_2,d)=\left(
\begin{array}{cccc}
Id_V & 0 & 0 & 0\\
0 & 0 & Id_W & 0\\
f(x,y_2+\xi) & 0 & 0 & 0 \\
\rho g_{\xi}(x,y_1,y_2)(d) & 0 & 0 & 0
\end{array}
\right)
\end{align*}
where $\rho$ is an arbitrary real number greater than $1$ fixed for the remainder of the section.

\begin{thm}[Universal Limit Theorem, inhomogeneous case]\label{_thm_ULT}
Let $k\ge 1$ be an integer, $V$ and $W$ Banach spaces, such that $V=V^1\oplus\cdots\oplus V^k$ for some Banach spaces $V^1,\dots, V^k$. 
Let $\Pi=(p_1,\dots,p_k)$ denote a $k$-tuple and  $\mfn{X}{}\in G\Omega_{\Pi}(V)$ be a geometric $\Pi$-rough path and $\xi$ an element in $W$. 

Suppose that there exist real numbers $\gamma_1$,\dots, $\gamma_{k+3}$ such that $\gamma_i>1-1/p_i$ for $i=1,\dots k$ and $\gamma_{k+j}>1-1/p_{\max}$ for $j=1,2,3$, furthermore the functions $h_0$, $h_1$ and $h_2$ are $\text{\protect{\textup{Lip}}}^{\Gamma_0,\Pi_0}$, $\text{\protect{\textup{Lip}}}^{\Gamma_1,\Pi_1}$ and $\text{\protect{\textup{Lip}}}^{\Gamma_2,\Pi_2}$ one-forms respectively for 
$\Gamma_0=(\gamma_1,\dots,\gamma_{k+1})$, $\Pi_0=(p_1,\dots,p_k,p_{\max})$, 
$\Gamma_1=(\gamma_1,\dots,\gamma_{k+2})$, $\Pi_1=(p_1,\dots,p_k,p_{\max},p_{\max})$,
and $\Gamma_2=(\gamma_1,\dots,\gamma_{k+3})$, $\Pi_2=(p_1,\dots,p_k,p_{\max},p_{\max},p_{\max})$.

Then the equation 
\begin{equation}
\dd\mfn{Y}{t}=f(\mfn{X}{t},\mfn{Y}{t})\dd \mfn{X}{t}, \ \mfn{Y}{0}=\xi\label{_eq_ULT_RDE}
\end{equation}
has a unique solution.
\end{thm}

The proof of Theorem \ref{_thm_ULT} is based on the proof of Lyons' Universal Limit Theorem in \cite{rp_notes}. 
We start with adapting some lemmas used in the original proof.
\begin{lem}\label{_lem_ULT_proof_1}
Let the Banach space $V$ be of the form $V=V^1\oplus\cdots\oplus V^k$ for some Banach spaces $V^1,\dots, V^k$.
Let $\Pi=(p_1,\dots,p_k)$ denote a $k$-tuple, $\eps>0$, and let $\omega$ be a control function.

Consider $\mfn{Z}{}=(\mfn{X}{},\mfn{Y}{})\in G\Omega_{\Pi\ast\Pi}(V\oplus V)$ and let $\mfn{W}{}\in G\Omega_{\Pi\ast\Pi}(V\oplus V)$ be the image of $\mfn{Z}{}$ under the linear map $(x,y)\to(x,\frac{y-x}{\eps})$. Assume that the $\Pi\ast\Pi$-variation of $\mfn{W}{}$ is controlled by $\omega$. Then there exists a constant $C$ depending only on $\Pi$, $\omega(0,T)$ and $\beta$, such that
\begin{equation*}
\left\Vert\pi_R\left(\mfn{X}{s,t}-\mfn{Y}{s,t}\right)\right\Vert \le C(\eps+\eps^{\Vert R\Vert}) \omega(s,t)^{\Vert R\Vert/p_{max}}, \ \forall (s,t)\in\Delta_T, \ \forall R\in\APi{1}.
\end{equation*}
\end{lem}
\begin{proof}
The claim is equivalent to Lemma 5.6 of \cite{rp_notes} adapted to the inhomogeneous smoothness case and the proof is analogous to the proof of the referred lemma. 

Let $R=(r_1,\dots,r_l)\in\APi{1}$ and $(s,t)\in\Delta_T$.
First, assuming that $\mfn{Z}{}=(\mfn{X}{},\mfn{Y}{})\in V\oplus V$ has bounded variation using the notation introduced in equation \eqref{_eq_iterated_int} and writing $\mfn{Y}{}=\mfn{X}{}+\eps\frac{\mfn{Y}{}-\mfn{X}{}}{\eps}$, we get
\begin{equation*}
\mfnc{Y}{R}{s,t}=\mfnc{X}{R}{s,t}+\sum_{
\begin{smallmatrix}
k_1,\dots,k_l\in\{0,1\}\\
k_1+\cdots+k_l>0
\end{smallmatrix}
}\eps^{k_1+\cdots+k_l}\mfnc{W}{(r_1+k_1*l,\dots,r_l+k_l*l)}{s,t}.
\end{equation*}

The assertion is implied by the continuity in the $\Pi\ast\Pi$-variation topology and by the control on $\mfn{W}{}$.
\end{proof}

\begin{lem}[Scaling Lemma, inhomogeneous version]\label{_lem_scaling}
Let the Banach space $V$ be of the form $V=V^1\oplus\cdots\oplus V^k$ for some Banach spaces $V^1,\dots, V^k$.
Let $\Pi=(p_1,\dots,p_k)$ denote a $k$-tuple, let $\omega$ be a control function and let $M\ge 1$ be a real number. Let $E=V^1\oplus\cdots\oplus V^l$ and $F=V^{l+1}\oplus\cdots\oplus V^k$ be Banach spaces.
Let $\Pi_1=(p_1,\dots,p_l)$ and $\Pi_2=(p_{l+1},\dots,p_k)$ denote the corresponding $l$ and $(k-l)$-tuples. 

Let $\mfn{Z}{}=(\mfn{X}{},\mfn{Y}{}):\Delta_T\to \TPi{1}(V)$ be a geometric $\Pi$-rough path such that
\begin{enumerate}[(i)]
\item the $\Pi$-variation of $\mfn{Z}{}$ is controlled by $M\omega$,
\item the $\Pi_1$-variation of $\mfn{X}{}=\hat{\pi}_E(\mfn{Z}{})$ is controlled by $\omega$,
\item $\mfn{Y}{}=\hat{\pi}_F(\mfn{Z}{})$.
\end{enumerate}

Then, for all $0\le \eps \le M^{-s_{m^*}}$,
the $\Pi$-variation of $(\mfn{X}{},\eps\mfn{Y}{})$ is controlled by $\omega$, where
\[
s_{m^*}=\max_{s_m\le 1}\left\{s_m\in S^{\Pi}\right\}=\max_{R\in\APi{1}}\deg_{\Pi}(R).
\]
\end{lem}

\begin{proof}
This lemma is analogous to Lemma 5.8 of \cite{rp_notes}, adapted to the inhomogeneous smoothness case. 

Let $\mfn{W}{}\in G\Omega_{\Pi}(V)$ denote the image of $\mfn{Z}{}$ under the linear map $(x,y)\to(x,\epsilon y)$. For a multi-index $R=(r_1,\cdots,r_m)$, let $|R|_F$ denote the cardinality of the set $\{r|\ r\in R, r>l\}$. Then if $\mfn{Z}{}$ has bounded variation, by simple rescaling arguments we get
\begin{equation*}
  \mfnc{W}{R}{s,t}=\eps^{|R|_F}\mfnc{Z}{R}{s,t}.
\end{equation*}
By continuity, the last equality holds for general geometric $\Pi$-rough path $\mfn{Z}{}$. 
This following inequality is now implied and completes the proof:
\begin{equation*}
\left\Vert\pi_R\left(\mfn{W}{s,t}\right)\right\Vert\le 
\eps^{|R|_F}M^{\deg_{\Pi}(R)}
\frac{\omega(s,t)^{\degg_{\Pi}(R)}}{\beta^k\Gamma_{\Pi}(R)}.
\end{equation*}

\end{proof}

Given the one-forms $h_i$, $i=1,2,3$, we define the following sequences of rough paths
\begin{equation*}
\mfn{Z}{0}(0)=(\mfn{X}{},0), \text{ and } \mfn{Z}{0}(n+1)=\int h_0(\mfn{Z}{0}(n))d h_0(\mfn{Z}{0}(n)), 
\end{equation*}

\begin{equation*}
\mfn{Z}{1}(0)=(\mfn{X}{},0,\mfn{Y}{}(1)), \text{ and } \mfn{Z}{1}(n+1)=\int h_1(\mfn{Z}{1}(n))d h_1(\mfn{Z}{1}(n)), 
\end{equation*}

\begin{equation*}
\mfn{Z}{2}(0)=(\mfn{X}{},0,\mfn{Y}{}(1),\mfn{Y}{}(1)), \text{ and } \mfn{Z}{2}(n+1)=\int h_2(\mfn{Z}{2}(n))d h_2(\mfn{Z}{2}(n)), 
\end{equation*}
for $n=0,1,\dots$, where $\mfn{Y}{}(n)=\hat{\pi}_W(\mfn{Z}{0}(n))$.

The definition of the above iterations imply the following lemma. 

\begin{lem}\label{_lem_iterations}
For all $n\ge 0$,
\begin{eqnarray*}
\mfn{Z}{0}(n)&=&(\mfn{X}{},\mfn{Y}{}(n))\\
\mfn{Z}{1}(n)&=&(\mfn{X}{},\mfn{Y}{}(n),\mfn{Y}{}(n+1))\\
\mfn{Z}{2}(n)&=&(\mfn{X}{},\mfn{Y}{}(n),\mfn{Y}{}(n+1),\rho^n(\mfn{Y}{}(n+1)-\mfn{Y}{}(n))).
\end{eqnarray*}
Furthermore, if the $\Pi$-variation of $\mfn{X}{}$ is controlled by $\omega$, then the $\Pi_i$-variation of $\mfn{Z}{i}(0)$ is controlled by $M\omega$ for $i=1,2$ respectively on $[0,T_{\rho}]$, 
where $M$ and $T_{\rho}$ are 
 defined below. 
\end{lem}

Recall the definitions 
$\Gamma_0=(\gamma_1,\dots,\gamma_{k+1})$, $\Pi_0=(p_1,\dots,p_k,p_{\max})$, $\Gamma_2=(\gamma_1,\dots,\gamma_{k+3})$, $\Pi_2=(p_1,\dots,p_k,p_{\max},p_{\max},p_{\max})$, furthermore we define 
$\Gamma_1=(\gamma_1,\dots,\gamma_{k+2})$, $\Pi_1=(p_1,\dots,p_k,p_{\max},p_{\max})$. 
By Theorem \ref{_thm_BoundOnInt}, there exists a constant $M_i$ depending only on $\Pi_i$, $\Gamma_i$, $\hat{\Gamma_i}$ and polynomially on the $\text{\protect{\textup{Lip}}}^{\Gamma_i,\Pi_i}$-norm of $h_i$, such that if $\mfn{Z}{i}$ is a rough path in the appropriate space with $\Pi_i$-variation controlled by some control $\omega$ such that $\omega(0,T)<1$, then the $\Pi_i$-variation of $\int h_i(\mfn{Z}{i})d\mfn{Z}{i}$ is controlled by $\omega$ for $i=0,1,2$ respectively. We define $M=\max(M_0,M_1,M_2)$, and without loss of generality we assume that $M\ge 1$. We chose $\eps=M^{-s_{m^*}}$.

Let $\omega_0$ be a control of the $\Pi$-variation of $\mfn{X}{}$. Let $T_{\rho}>0$ be chosen to satisfy $\omega_0(0,T_{\rho})=\eps^{p_{\max}}$. 
Note that for $R\in\APi{1}$,
\begin{equation*}
1\ge \degg_{\Pi}(R) = \sum_{i=1}^k\frac{n_j(R)}{p_i}\ge \sum_{i=1}^k\frac{n_j(R)}{p_{\max}} =\frac{\Vert R\Vert}{p_{\max}}.
\end{equation*}
This implies that by setting $\omega=\eps^{-p_{\max}}\omega_0$, $\eps^{-1}\mfn{X}{}$ is controlled by $\omega$ and $\omega(0,T_{\rho})\le 1$. 

\begin{lem}\label{_lem_ULT_proof_4}
For all $n\ge 0$, the $\Pi_0$, $\Pi_1$ and $\Pi_2$-variation of the following rough paths respectively
\begin{eqnarray*}
&&(\eps^{-1}\mfn{X}{},\mfn{Y}{}(n))\\
&&(\eps^{-1}\mfn{X}{},\mfn{Y}{}(n),\mfn{Y}{}(n+1))\\
\text{and}&&(\eps^{-1}\mfn{X}{},\mfn{Y}{}(n),\mfn{Y}{}(n+1),\rho^n(\mfn{Y}{}(n+1)-\mfn{Y}{}(n)))
\end{eqnarray*}
are controlled by $\omega$ on $[0,T_{\rho}]$.
\end{lem}
The proof is based on the Scaling lemma \ref{_lem_scaling} and analogous to the proof of Proposition 5.9 in \cite{rp_notes}.

Now we prove the main theorem. We follow the proof of the Universal Limit Theorem corresponding to the homogeneous case presented in \cite{rp_notes}.

\begin{proof}[Proof of Theorem \ref{_thm_ULT}:]
By Lemma \ref{_lem_ULT_proof_4}, the $\Pi_2$-variation of $\mfn{Z}{2}(n)$ for all $n\ge 0$ is controlled by $\omega$ on $[0,T_{\rho}]$. We define the linear map $A:V\oplus W\oplus W\oplus W\to (V\oplus W)\oplus (V\oplus W)$ by
\begin{equation*}
A(x,y_1,y_2,d)=((x,y_1),(0,d)).
\end{equation*}
This linear map has norm 1. Note that 
\begin{equation*}
A(\mfn{Z}{2}(n))=
((\mfn{X}{},\mfn{Y}{}(n)),\rho^n(0,\mfn{Y}{}(n+1)-\mfn{Y}{}(n)))=
((\mfn{X}{},\mfn{Y}{}(n)),\rho^n[(\mfn{X}{},\mfn{Y}{}(n+1))-(\mfn{X}{},\mfn{Y}{}(n))])
\end{equation*}
is controlled by $\omega$ on $[0,T_{\rho}]$. 
Then Lemma \ref{_lem_ULT_proof_1} and Lemma \ref{_lem_ULT_proof_4} imply the existence of a constant $C$ depending only on $\Pi$, $\omega(0,T)$ and $\beta$, such that for all $(s,t)\in\Delta_T$
\begin{eqnarray}
\left\Vert
\pi_R\left(((\mfn{X}{},\mfn{Y}{}(n))_{s,t}-
(\mfn{X}{},\mfn{Y}{}(n+1))_{s,t}\right)
\right\Vert& \le &
C
\rho^{-n}\omega(s,t)^{\Vert R\Vert/p_{\max}}, \ \forall R\in\APig{\hat{\Pi}}{1}
.\nonumber\\
\label{_eq_main_ineq}
\end{eqnarray}
The inequality implies that $(\mfn{X}{},\mfn{Y}{}(n))$ converges in the $\Pi_0$-variational topology on the interval $[0,T_{\rho}]$ to a rough path $(\mfn{X}{},\mfn{Y}{})\in G\Omega_{\Pi_0}$, which is also a solution to the RDE \eqref{_eq_ULT_RDE}. 

Note that once $\rho$ is chosen, $T_{\rho}$ is bounded from below where the bound only depends on the Lip-norm of $h_0$, $h_2$, $\Pi$, $\Gamma_2$ and the modulus of continuity of $\omega$ on $[0,T]$. This implies that one can paste together local solutions in order to get a solution on the whole interval $[0,T]$. 

In order to prove uniqueness, we assume that $\wh{\mfn{Z}{}}=(\mfn{X}{},\wh{\mfn{Y}{}})$ is also a solution to the RDE \eqref{_eq_ULT_RDE}. We compare $\mfn{Y}{}(n)$ and $\wh{\mfn{Y}{}}$ by defining the function $h_3:V\oplus W\oplus W\oplus W\to \text{End}(V\oplus W\oplus W\oplus W)$ by
\begin{equation*}
h_3(x,y,\wh{y},\wh{d})=\left(
\begin{array}{cccc}
Id_V & 0 & 0 & 0 \\
f(y+\xi) & 0 & 0 & 0 \\
0 & 0 & Id_W & 0 \\
\rho g_{\xi}(y,\wh{y})(\wh{d}) & 0 & 0 & 0
\end{array}
\right)
\end{equation*}
and defining $\mfn{Z}{3}(n)$ by
\begin{equation*}
\mfn{Z}{3}(0)=(\mfn{X}{},0,\wh{\mfn{Y}{}},\wh{\mfn{Y}{}}), \text{and } \mfn{Z}{3}(n+1)=\int h_3(\mfn{Z}{3}(n)).
\end{equation*}
Arguments analogous to the proof of Lemma \ref{_lem_iterations} (ref. \cite{rp_notes}) imply that
\begin{equation*}
\mfn{Z}{3}(n)=(\mfn{X}{},\mfn{Y}{}(n),\wh{\mfn{Y}{}},\rho^n(\wh{\mfn{Y}{}}-\mfn{Y}{}(n))).
\end{equation*}
Now analogously to Lemma \ref{_lem_ULT_proof_4}, the $\Pi_2$-variation of $\mfn{Z}{3}(n)$ is controlled by $\omega$ on a small enough interval. Then by Lemma \ref{_lem_ULT_proof_1}, $\mfn{Y}{}=\wh{\mfn{Y}{}}$ on the same interval. The uniqueness of $\mfn{Y}{}$ is implied by the uniform continuity of $\omega$.

Define $I_f(\mfn{X}{},\xi)=(\mfn{X}{},\mfn{Y}{})$. Analogous arguments to the proof of the Universal Limit Theorem in \cite{rp_notes} imply that $I_f$ is continuous from $G\Omega_{\Pi}(V)\times W\to G\Omega_{\Pi_0}(V\oplus W)$ in the $\Pi$-$\Pi_0$-variation topology. 
\end{proof}


\subsection*{Acknowledgements}
The author is grateful to Terry Lyons,  Dan Crisan, Ben Hambly, Peter Friz and Michael Caruana for  the  valuable and useful comments and suggestions.



\begin{thebibliography}{99}
\bibitem{fbmQian} L. Coutin \& Z. Qian, 2002, \textit{Stochastic analysis, rough paths analysis and fractional Brownian motion}, Probab. Theory Relat. Fields 122, pp. 108-140

\bibitem{bmFrizVictoir} P. Friz \& N. Victoir, 2005, \textit{Approximations of the Brownian rough path with applications to stochastic analysis}, Annales de l'Institut Henri Poincar\'e - Prob. et Stat, Vol 41, pp. 703-724

\bibitem{gaussFrizVictoir} P. Friz \& N. Victoir, \textit{Differential Equations Driven by Gaussian signals}, Annales de l'Institut Henri Poincar\'e - Prob. et Stat, Vol 46, pp. 369-413

\bibitem{bookFrizVictoir} P. Friz \& N. Victoir, \textit{Multidimensional Stochastic Processes as Rough Paths}, Cambridge University Press, 2010

\bibitem{lggy_thesis} L.G. Gyurk\'{o}, 2009 \textit{Numerical methods for approximating solutions to Rough Differential Equations}, DPhil thesis, University of Oxford


\bibitem{pqrp} A. Lejay \& N. Victoir, 2006, \textit{On (p,q)-rough paths}, J. Different. Equations, 225(1), pp. 103-133

\bibitem{rp_paper} T.J. Lyons, 1998, \textit{Differential equations driven by
        rough signals}, Revista Mathematica Iber. Vol 14, Nr. 2,
      pp. 215-310

\bibitem{rp_book} T.J. Lyons \& Z. Qian, 2002, \textit{System Control and Rough Paths},
      Oxford mathematical monographs, Clarendon Press, Oxford

\bibitem{rp_notes} T.J. Lyons, M. Caruana, T. L\'evy, 2007, \textit{Differential Equations Driven by Rough Paths}, Ecole d'Et\'e de Porbabilit\'es de Saint-Flour XXXIV - 2004, Lecture Notes in Mathematics, Springer

\bibitem{fbmSole} A. Mille \& M. Sanz-Sol\'e, \textit{Approximation of rough paths of fractional Brownian motion}, In \textit{Seminar on Stochastic Analysis Random Fields and Applications} 275-303, Progress in Probability 59. Birkh\"auser, Basel, 2008


\end{thebibliography}
\end{document}